\documentclass[11pt]{article}
\usepackage[english]{babel}
\usepackage[a4paper, left=3.4cm, right=3.4cm, top=3cm]{geometry}

\usepackage{xfrac}

\usepackage{mathptmx}      

\usepackage{amsmath,amsthm}
\usepackage{amssymb}
\usepackage{amsfonts}
\usepackage{oldgerm}
\usepackage{mathrsfs}
\usepackage{comment}
\usepackage{tikz}
\usepackage{pgfplots}
\usepackage[scanall]{psfrag}
\usepackage{graphicx,pst-all,bm}
\usepackage{flushend}
\usepackage{subfig}
\usepackage{mdwlist}
\usepackage{multirow}
\usepackage{color}

\usepackage{longtable}
\usepackage{cuted}
\usepackage{verbatim}

\usepackage{siunitx}

\newcommand{\N}{\mathbb{N}}

\newcommand{\R}{\mathbb{R}}

\DeclareMathOperator{\im}{im \, }

\newcommand{\Real}{\mathbb{R}}

\newcommand{\rank}{\operatorname{rank}}
\newtheorem{proposition}{Proposition}[section]

\newtheorem{lemma}[proposition]{Lemma}

\newtheorem{conjecture}[proposition]{Conjecture}

\pgfplotsset{compat=1.17}
\usepackage{authblk}

\begin{document}

\title{Regular linear time varying DAEs are equivalent to DAEs in strong standard canonical form}

\author[1]{Diana Est\'evez Schwarz}
\author[2]{Ren\'e Lamour}
\author[2]{Roswitha M\"arz}

\affil[1]{\small Berliner Hochschule f\"ur Technik}
\affil[2]{\small Humboldt University of Berlin, Institute of Mathematics}

\maketitle

\begin{abstract}
The relationship between solvability of linear diffential-algebraic equations (DAEs) and their transformability  into canonical forms  has been  investigated for more than forty years. 

After a comparative analysis of numerous DAE frameworks the notions regularity and almost regularity were established only recently.  Regular DAEs resulted to be equivalently transformable into so-called standard canonical forms (SCF) with block-structured nilpotent matrix functions featuring certain rank properties.

In this paper we prove that for regular DAEs,  even a transformation into a strong standard canonical form (SSCF) is possible, i.e. a SCF with a constant nilpotent matrix. We start from block-structured SCF and give a constructive proof.

\end{abstract}

\setcounter{secnumdepth}{3}
\setcounter{tocdepth}{3}

\section{Introduction}\label{sec:Introduction}
We focus on linear differential algebraic equations (DAEs) in standard form,
\begin{align}\label{1.DAE}
 Ex'+Fx=q,
\end{align}
in which $E,F:\mathcal I\rightarrow \Real^{m\times m}$ are sufficiently smooth, at least continuous, matrix functions on the interval $\mathcal I\subseteq\Real$. We assume $\im \begin{bmatrix}
	E(t) & F(t)
\end{bmatrix} = \Real^m$  and a constant $r=\rank E(t)<m$ for all $t\in\mathcal I$ . 
\bigskip

Starting from results from \cite{HaMae2023}, in \cite{commonground2024} the equivalence of thirteen different characterizations of regular $\{E, F \}$ are given. 
The common ground of these 
statements is the existence of  the index $\mu$ and the canonical characteristics
\begin{align}\label{char}
 r<m,\; \theta_{0}\geq\cdots\geq\theta_{\mu-2}>\theta_{\mu-1}=0,\; d=r-\sum_{k=0}^{\mu-2}\theta_k,
\end{align}
that are constant, persist under equivalence transformations and can be computed using different index frameworks, cf.\ \cite{commonground2024}.
\medskip

A pair of the form
\begin{eqnarray*}
	\left\{ \begin{bmatrix}
	I_d & 0 \\
	0 & N
\end{bmatrix}, \begin{bmatrix}
	\Omega & 0 \\
	0 & I_{m-d}
\end{bmatrix} \right\} , \quad \Omega: \mathcal I\rightarrow\Real^{m-d\times m-d}, \quad N : \mathcal I\rightarrow\Real^{d\times d}, \quad d \in \N,
\end{eqnarray*}
where $N$ is pointwise nilpotent and lower or upper triangular, is said to be in standard canonical form (SCF), cf.\ \cite{CaPe83}\footnote{In \cite{CaPe83} it is emphasized that $N$ does not necessarily have constant rank.}$^{,}$\footnote{An algorithm to determine whether real analytic $\{E,F\}$ are transferable in SCF or not can be found in \cite{BergerIlchmann}.}. If in addition $N$ is constant, then the system is in strong standard canonical form (SSCF).

\medskip
In this paper, we  proof the following conjecture made in \cite{commonground2024}, obtaining a fourteenth equivalent characterization of regularity that means transformability into a strong standard canonical form.
We will see that the constant matrix $N$ is nilpotent with index $\mu$ and
\[
\theta_i=\rank N^{i+1} - \rank N^{i+2}, \quad i=0, \ldots, \mu-2, \quad d=r-\sum_{i=0}^{\mu-2} \theta_i.
\]
These canonical characteristics are particularly easy to recognize in nilpotent matrices in Jordan normal form, see Figure \ref{fig:PowersJ}.

\begin{conjecture} \cite{commonground2024}\label{conjecture} Let $E, F:\mathcal I\rightarrow\Real^{m\times m}$ be sufficiently smooth. If the pair $\{E,F\}$ is regular with index $\mu\geq2$ then it is transformable into strong standard canonical form and the Jordan normal form of the matrix $N$ is exactly made of
 \begin{align*}
  &m-r-\theta_0 \quad \text{Jordan blocks of order } 1,\\
  &\theta_0-\theta_1 \quad \text{Jordan blocks of order } 2,\\
  &\theta_1-\theta_2 \quad \text{Jordan blocks of order } 3,\quad \\
	&...\\
   &\theta_{\mu-3}-\theta_{\mu-2}\quad \text{Jordan blocks of order } \mu-1,\\
  &\theta_{\mu-2}\quad \text{Jordan blocks of order } \mu.
 \end{align*}
 \end{conjecture}

\begin{figure}
\includegraphics[width=\textwidth]{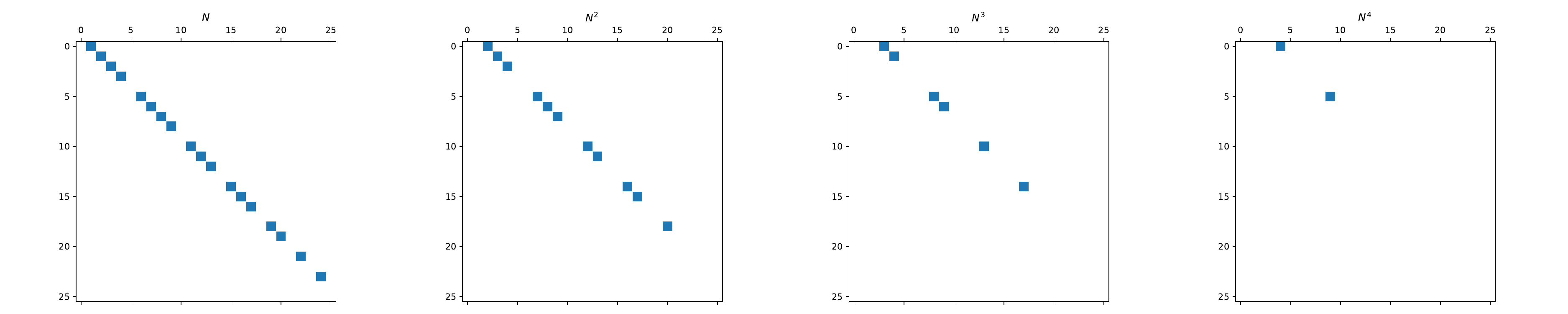}
\caption{Powers of $N$ with $N^5=0$ for $N$ consisting of $7$ Jordan blocks, $m=26$, $r=\rank N=18$, $d=0$, $\theta_0= 7,  \theta_1= 5,  \theta_2= 4,  \theta_3=2$. This corresponds to two blocks of order 5, two blocks of order 4, one block of order 3, two blocks order 2 and one block order 1. }
\label{fig:PowersJ}
\end{figure}

The paper is organized as follows. In Section \ref{sec:EquivRel} we deepen the equivalence concept for DAEs and pairs of matrix functions $\{E,F\}$ with two lemmas that are decisive for the following constructive proof.  In Section \ref{sec:Reg+BSSCF}
we introduce block-structured standard canonical forms and summarize the results from \cite{commonground2024} which describe their relationship with the canonical characteristics.
\medskip

For the proof of the conjecture for the nilpotent part of block-structured standard canonical forms two different variants are possible. We describe one variant in detail in Section \ref{sec:cols} and briefly outline the most important differences of the second in Section \ref{sec:rows}.
\medskip

In Section \ref{sec:sum} we summarize the results.

\section{Equivalence Relations}\label{sec:EquivRel}
Two pairs of matrix functions $\{E,F\}$ and $\{\tilde E,\tilde F\}$, and also the associated DAEs, are called \textit{equivalent}\footnote{In the context of the strangeness index \textit{globally equivalent}, e.g. \cite[Definition 2.1]{KuMe1996}, and \textit{analytically equivalent} in \cite[Section 2.4.22]{BCP89}. }, if there exist pointwise nonsingular, sufficiently smooth matrix functions $L, K:\mathcal I\rightarrow \Real^{m\times m}$, such that
\begin{align}\label{1.Equivalence}
 \tilde E=LEK,\quad \tilde F=LFK+LEK'.
\end{align}
An equivalence transformation 
 goes along with  the premultiplication of \eqref{1.DAE} by $L$ and the coordinate change $x=K\tilde x$ resulting in the further DAE $\tilde E\tilde x'+\tilde F\tilde x=Lq$.
Recall that \eqref{1.Equivalence} actually defines a reflexive, symmetric, and transitive equivalence relation $\{E,F\}\sim \{\tilde E,\tilde F\}$.
\medskip

We will use the following elementary results later on several times.

\begin{lemma}\label{lemma_k_inner}
For a pair of matrix functions $\{E,F\}$ and any differentiable matrix function $K:\mathcal I\rightarrow \Real^{m\times m}$
 that is pointwise nonsingular on $\mathcal I$ and leads to a pointwise nonsingular matrix function 
\[
G:=FK+EK'
\]
it holds
\[
\{E,F\} \quad \sim  \quad \{\hat{E},I\},
\]
whereas $I$ is the identity matrix, $L:=G^{-1}=(FK+EK')^{-1}$, and $\hat{E}:=L E K $.
\end{lemma}

\begin{proof}
Let us first motivate the definition of $L$.
For any nonsingular $K$ and $y:=K^{-1}x$ we can reformulate \eqref{1.DAE} as follows
\begin{eqnarray*}
(E K)K^{-1}x'+Fx &=& (EK) ((K^{-1}x)'-(K^{-1})'x)+Fx= (EK)y'+(F-EK(K^{-1})')Ky\\
&=& (EK)y'+ (FK+EK')y,
\end{eqnarray*}
using $K(K^{-1})'K =- K'$. For $L: = (FK+EK')^{-1}$ we can therefore reformulate the DAE \eqref{1.DAE}  in terms of $y$ as
\begin{align*} 
 LEKy'+ y= Lq,
\end{align*}
i.e. 
\[
\hat E=LEK,\quad \hat F=I.
\]
It is also easy to verify directly that  $\{\hat E,I \}$ is equivalent to $\{ E, F\}$
\begin{eqnarray*}
 L^{-1} \hat E  K^{-1} &=& E,\\
 L^{-1} \hat F K^{-1}+ L^{-1} \hat E  (K^{-1})' &=& 
(FK+EK') I  K^{-1} +  E K (K^{-1})' \\
&=& F+EK'K^{-1} + EK(K^{-1})'=F.
\end{eqnarray*}
\end{proof}

\begin{lemma}\label{lemma_k_h_l}
For a matrix functions $E$  and any differentiable matrix function $K:\mathcal I\rightarrow \Real^{m\times m}$
 that is pointwise nonsingular and differentiable on $\mathcal I$  and leads to a pointwise nonsingular matrix function 
\[
H:=I+K^{-1}EK'
\]
it holds
\[
\{E, I \} \quad \sim \quad \{\hat E, I \}
\]
for  $L=H^{-1}K^{-1}=(I+K^{-1}EK')^{-1}K^{-1}=(K+EK')^{-1}$ and $\hat E = LEK$. Obviously, if $K^{-1}EK'$ is strictly upper or strictly lower triangular, then  $H$ is nonsingular by construction. 
\end{lemma}

\begin{proof}
The assertion follows directly from Lemma \ref{lemma_k_inner} and $KH=K(I+K^{-1}EK')=K+EK'=L^{-1}$.
\end{proof}

\section{ Regularity and block-structured
 standard canonical forms} \label{sec:Reg+BSSCF}
In \cite{commonground2024}, thirteen equivalent definitions for regularity of linear DAEs were given. Here, we focus only on those that involve block-structured
 standard canonical forms and recall that for given integers $\mu\geq 2,  \ell_1,\ldots,\ell_{\mu}$, 
we denote by $SUT=SUT(\mu,\ell_1,\ldots,\ell_{\mu})$ the set of all strictly upper triangular  matrix functions showing the block structure
\begin{align*}
N=\begin{bmatrix}
   0&N_{12}&*&\cdots&*\\
   &0&N_{23}&*&*\\
   &&\ddots&\ddots&\vdots\\
   &&&&N_{\mu-1, \mu}\\
   &&&&0
   \end{bmatrix}, \quad N_{ij}=(N)_{ij}:\mathcal I\rightarrow \Real^{\ell_i\times \ell_j},\quad N_{ij}=0 \quad\text{for}\quad i\geq j.
	\end{align*}
We refer to \cite{CRR,commonground2024} for the properties of $SUT$ and focus on
the following two subsets that are of special interest to describe crucial properties of DAEs.

\textbullet\; Supposing $\ell_1\geq \cdots\geq \ell_{\nu}$ we denote by $SUT_{column}\subset SUT$ the set  of all $N\in SUT$ having exclusively blocks $(N)_{i,i+1}$ with full column rank\footnote{This structure is figured out in the context of the tractability index, \cite{CRR}.}, that is 
\begin{align}\label{N.col}
 \rank (N)_{i,i+1}=\ell_{i+1},\quad i=1,\ldots,\nu-1.
\end{align}

\textbullet\; Supposing $\ell_1\leq \cdots\leq \ell_{\nu}$ we denote by $SUT_{row}\subset SUT$ the set  of all $N\in SUT$ having exclusively blocks $(N)_{i,i+1}$ with full row rank\footnote{This structure is figured out in the context of the strangeness index, \cite{KuMe1994,KuMe2024}.}, that is 
\begin{align}\label{N.row}
 \rank (N)_{i,i+1}=\ell_{i},\quad i=1,\ldots,\nu-1.
\end{align}

Both structures are visualized in Figures \ref{fig:Powers_cols} and \ref{fig:Powers_rows}.

According to \cite{commonground2024}, two equivalent characterizations of regularity are possible, among others:

\begin{itemize}
	\item 
A matrix pair $\{E,F\}$ is regular, iff matrix functions $N_c \in SUT_{columns}$ exist such that 
\[
\{E,F\} \ \sim \ \left\{ \begin{bmatrix}
	I & 0 \\
	0 & N_c
\end{bmatrix},\begin{bmatrix}
	\Omega & 0 \\
	0 & I
\end{bmatrix} \right\}.
\]
and the blocks of $N_c$ have the sizes
\[
\ell_1=m-r, \quad \ell_{i+1}=\theta_{i-2}, \quad i=1, \ldots, \mu-1.
\]
\item A matrix pair $\{E,F\}$ is regular, iff  matrix functions $N_r \in SUT_{rows}$ exist such that
\[
\{E,F\} \ \sim \  \left\{ \begin{bmatrix}
	I & 0 \\
	0 & N_r
\end{bmatrix}, \begin{bmatrix}
	\Omega & 0 \\
	0 & I
\end{bmatrix} \right\}.
\]
and the blocks of $N_c$ have the sizes
\[
\ell_{\mu}=m-r, \quad \ell_i=\theta_{\mu-i-1}, \quad i=1, \ldots, \mu-1.
\]

\end{itemize}
\begin{figure}
\includegraphics[width=0.9\textwidth]{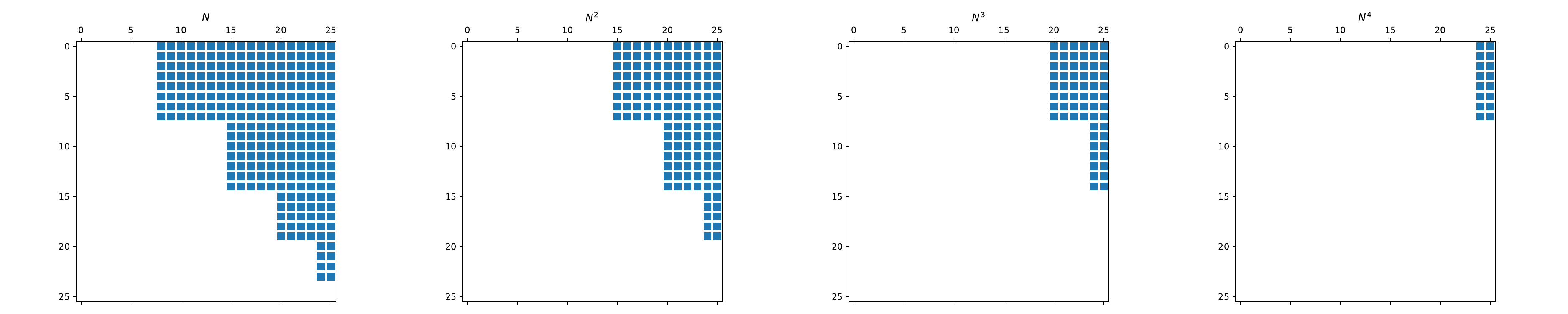}\\
\caption{Powers of $N$ with $N^5=0$ for $N$ having secondary blocks of full column rank, $m=26$,  $r=18$, $\theta_0= 7,  \theta_1= 5,  \theta_2= 4,  \theta_3=2$.}
\label{fig:Powers_cols}
\end{figure} 

\begin{figure}
\includegraphics[width=0.9\textwidth]{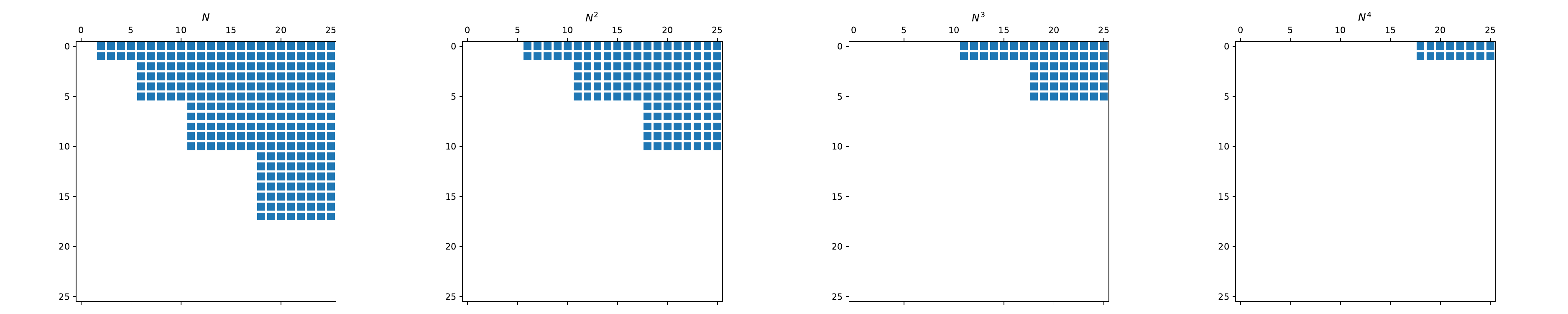}\\
\caption{Powers of $N$ with $N^5=0$ for $N$ having secondary blocks of full row rank, $m=26$,  $r=18$, $\theta_0= 7,  \theta_1= 5,  \theta_2= 4,  \theta_3=2$.}
\label{fig:Powers_rows}
\end{figure} 

For short, we write this as
\[
\{E,F\} \ \sim \ \left\{ \begin{bmatrix}
	I & 0 \\
	0 & N_c
\end{bmatrix}, \begin{bmatrix}
	\Omega & 0 \\
	0 & I
\end{bmatrix} \right\} \ \sim \ \left\{ \begin{bmatrix}
	I & 0 \\
	0 & N_r
\end{bmatrix}, \begin{bmatrix}
	\Omega & 0 \\
	0 & I
\end{bmatrix} \right\}.
\]
For the proof, we focus first on $\{N_c,I\}$ in Section \ref{sec:cols} and for completeness also on  $\{N_r,I\}$ in Section \ref{sec:rows}. In both cases we use equivalence transformations with matrices $L$, $K$, such that for the pure DAE part $\{N,I\}$ it holds
\[
\hat{N}= LNK, \quad LIK+LN(K)' =I , \quad \mbox{i.e.} \quad L=(K+NK')^{-1}=(I+K^{-1}NK')^{-1}K^{-1}
\]
and focus our attention on $N$ and $\hat{N}$ only. In all steps, $K^{-1}NK' \in SUP$ is given, implying invertible $K+NK'$, see Lemma \ref{lemma_k_h_l}. 
By a composed transformation, Conjecture \ref{conjecture} results straight forward.

\section{Proof with full-column rank assumption}\label{sec:cols}

If for integers $ \ell_{\mu} \leq \ldots \leq \ell_1$ we define elementary matrices from $SUT_{columns}$ by
\[
N^{(E_c)} =
	\begin{bmatrix}
   0&N^{(E_c)}_{12}&0&\cdots&0 \\
   &0&N^{(E_c)}_{23}& 0&0\\
   &&\ddots&\ddots&\vdots\\
   &&&&N^{(E_c)}_{\mu-1, \mu}\\
   &&&&0
   \end{bmatrix} \in \Real^{m\times m}
\]
with blocks
\begin{equation*}
N^{(E_c)}_{i,i+1} = \begin{bmatrix}
	I_{\ell_{i+1}} \\
	0
\end{bmatrix} \in \Real^{\ell_{i} \times \ell_{i+1}}. 
\end{equation*}
This structure is visualized in Figure \ref{fig:Powers_Ecols}.

\begin{figure}
\includegraphics[width=0.8\textwidth]{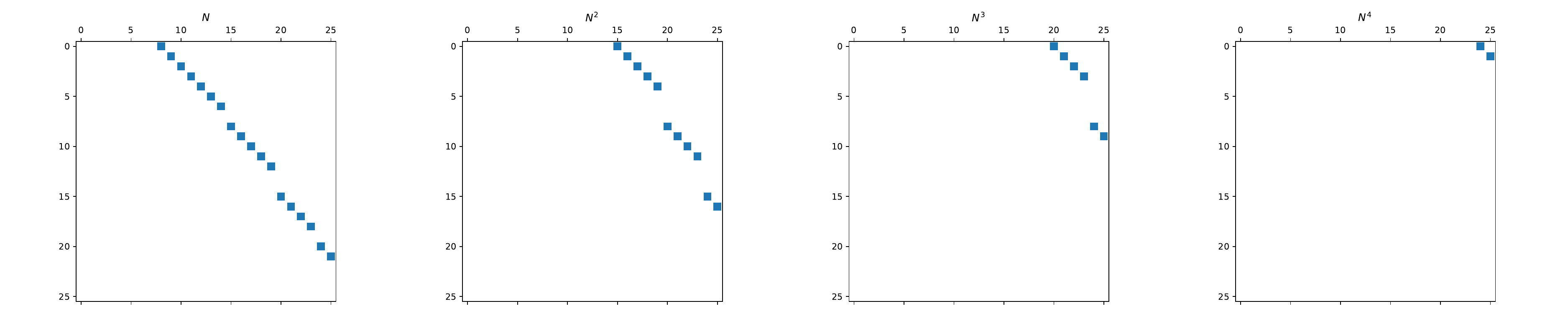}
\caption{Powers of $N=N^{(Ec)}$ with $N^5=0$ for $N$ with full column-rank secondary blocks, $m=26$,  $r=18$, $\theta_0= 7,  \theta_1= 5,  \theta_2= 4,  \theta_3=2$.}
\label{fig:Powers_Ecols}
\end{figure} 

\begin{proposition} \label{pro:N_col}
For every matrix function $N \in SUT_{columns}$ it holds

\[
\{N,I\} \quad \sim \quad \{N^{(E_c)},I \}.
\]
\end{proposition}

We want to emphasize already now that this proposition means the matrix pair $\{N,I \} $ for the matrix functions $N(t)$ is equivalent to a constant matrix pair in the sense of \eqref{1.Equivalence}. Therefore, it is equivalent to a pair where the constant nilpotent matrix can be assumed to have Jordan form.

\begin{proof}
For the proof, first we have to make a rearrangement in an initial step
\[
\{N,I\} \quad \sim \quad \{N^{(0)},I\}
\]
and perform then  a finite iteration
\[
  \{N^{(0)},I \}  \quad \sim \quad \{N^{(1)},I \}  \quad \sim \quad \ldots \quad \sim \quad \{N^{(\mu-1 )},I \}  \quad =    \quad \{N^{(E_c)},I\} \quad.
\] 
These steps, which use Lemma \ref{lemma_k_h_l}, are described in detail in the subsections below and  result in the assertion. The matrices $H$ from Lemma \ref{lemma_k_h_l} will be nonsingular by construction.
\end{proof}

\begin{figure}[htbp]
    \begin{minipage}[b]{\textwidth}
        \includegraphics[width=\textwidth]{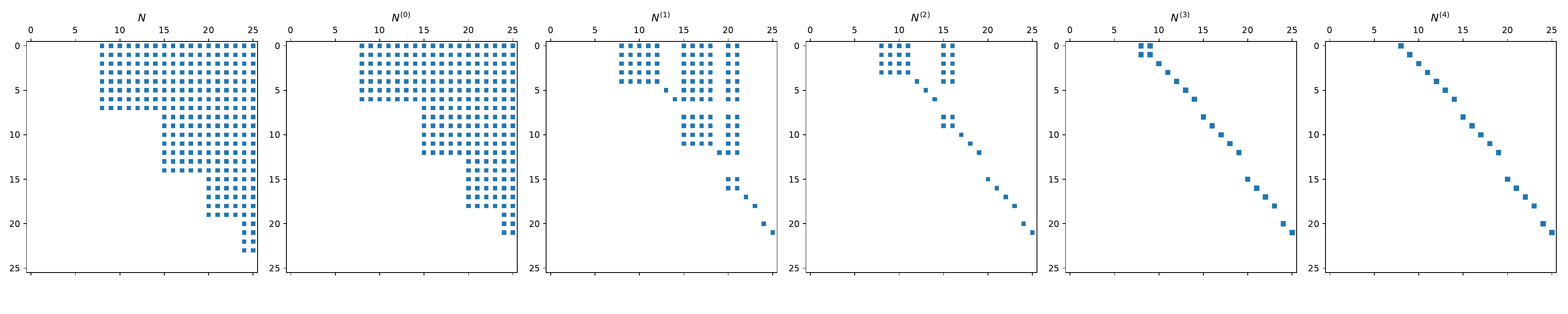}
        \caption{Visualization of  $N^{(k)}=L^{(k-1)}N^{(k-1)}K^{(k-1)}$ for $ \mu$ steps starting from $N$ (left) and ending with $N^{(Ec)}$ for  $N'\equiv 0$.}
        \label{fig:Step_dot_zero}
    \end{minipage}\\
		\begin{minipage}[b]{\textwidth}
        \includegraphics[width=\textwidth]{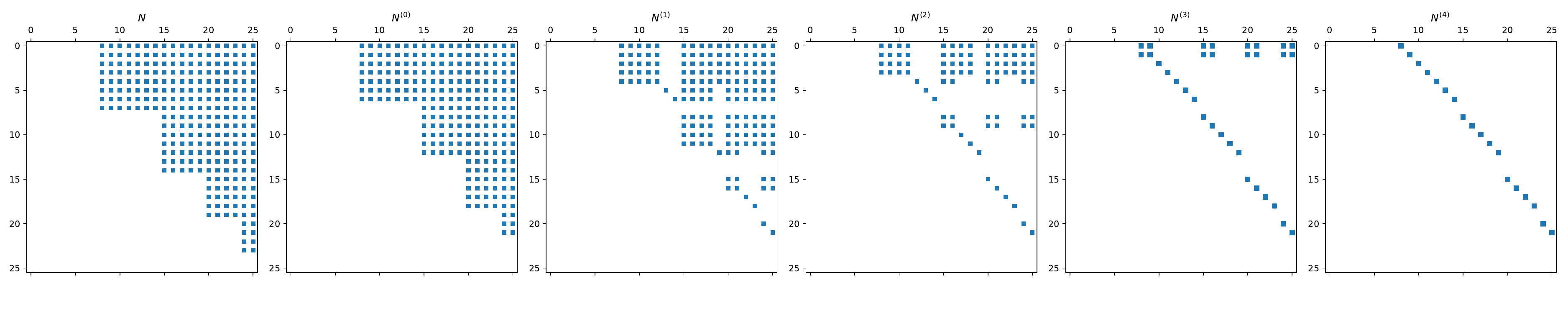}
        \caption{Visualization of  $N^{(k)}=L^{(k-1)}N^{(k-1)}K^{(k-1)}$ for $ \mu$ steps starting from $N$ (left) and ending with $N^{(Ec)}$ for $N'\neq 0$.}
        \label{fig:Steps_dot_not_zero}
    \end{minipage}
\end{figure}

\subsection{ \textbf{Step 0:} Obtaining quadratic nonsingular blocks by smooth SVDs}

The aim of this step is to construct a matrix function $N^{(0)} \in SUT_{columns}$ such that 
	\[
 \{N,I\} \quad \sim \quad \{N^{(0)},I \} \quad \mbox{and} \quad 	N^{(0)}_{i,i+1} =\begin{bmatrix}
		R_{i+1,i+1}\\
		0
	\end{bmatrix}:\mathcal I\rightarrow \Real^{\ell_{i+1}\times \ell_i }
	\]
	for nonsingular matrix functions $R_{i+1,i+1}$, $i=2, \ldots, \mu-1$ . \\

Using the singular value decomposition we define
\begin{itemize}
	\item $N_{i,i+1}  = U_{i} \Sigma_i V_{i}^T$. For the existence of such a decomposition we refer to e.g.\  \cite[Theorem 3.1.9]{KuMe2024}.
	\item $B_U$: Block diagonal matrix function containing $U_1, \ldots, U_{\mu-1}, I_{\ell_{\mu}}$ with $B_U^{-1}=B_U^T$.

\end{itemize}
Now we want to use Lemma \ref{lemma_k_h_l}. For a better understanding, let's retrace its main step here.
Multiplying 
$
N x' + x=q
$
with $B_U^T$ from the left leads to
\[
B_U^T N x' + B_U^Tx=B_U^Tq.
\]
For $y=B_U^Tx$, ($x=B_Uy$),  $H:=I+ B_U^T N (B_U)'$ we can reformulate this equation as:
\begin{eqnarray*}
B_U^T N (B_Uy)' + y&=&B_U^T N (B_U)'y + B_U^T N B_Uy' + y\\
&=& B_U^T N B_Uy' + (I+ B_U^T N (B_U)') y\\
&=& B_U^T N B_Uy' + H y.
\end{eqnarray*}
Since $B_U^T N (B_U)' \in SUT$, by construction $H$ from Lemma \ref{lemma_k_h_l} is nonsingular, and we can rewrite the DAE as
\[
H^{-1} B_U^T N B_Uy' +  y= H^{-1}B_Uq,
\]
i.e. for $\hat{N}=H^{-1} B_U^T N B_U$ it holds $\{N,I\} \sim \{\hat{N},I\}$ for $K=B_U$, $L=H^{-1} B_U^T$.

Due to the block-structure of $B_U$, it holds $\hat{N} \in  SUT_{columns}$, and
	\[
(\hat{N})_{i,i+1} = (B_U^T N B_U)_{i,i+1} = U_i^T U_{i} \Sigma_i V_{i}^T U_{i+1} = \Sigma_i V_{i}^T U_{i+1} =:\begin{bmatrix}
	R_{i+1,i+1} \\
	0
\end{bmatrix}
\]
for a nonsingular  matrix function $R_{i+1,i+1}$.

Setting $N^{(0)}:=\hat{N}\in SUT_{columns}$ we obtain the required pair
\[
\{ N^{(0)}, I \} \sim \{ N, I \}.
\]
Note further that by definition 
\begin{itemize}
	\item in $N$ the last $m-\sum_{i=1}^{\mu-1}\ell_i$ rows have to be zero,
	\item in $N^{(0)}$ the last $m-\sum_{i=1}^{\mu-1}\ell_i + (\ell_{\mu-1}-\ell_{\mu})$ rows have to be zero.
\end{itemize}

\subsection{ \textbf{Step 1:} Obtaining $N^{(E_c)}$  } \label{sec:Step1_col}
The aim of this step is  to construct $N^{(E_c)} \in SUT_{columns}$ by a finite iteration such that  
\[
 \quad \{N^{(0)},I \}  \quad \sim \quad \{N^{(1)},I \}  \quad \sim \quad \ldots \quad \sim \quad \{N^{(\mu-1 )},I \}  \quad =    \quad \{N^{(E_c)},I\}.
\]

For our purposes, we will use the following lemma several times.
\begin{lemma}\label{lem:construct_Kc}
	For any $N \in SUT_{columns}$, a corresponding $N^{(Ec)}\in SUT_{columns}$ and 
	\[
	K:=N(N^{(Ec)})^T + (I-N^{(Ec)}(N^{(Ec)})^T) 
	\]
	it holds
	\[
	K(N^{(Ec)})=N.
	\]
\end{lemma}
\begin{proof}
By construction we have
	\[
	 N^{(Ec)}(N^{(Ec)})^T =\begin{bmatrix}
	I_{\ell_2} & 0  & \cdots  & 0 &  0 &0\\
	0 & 0 & & & & \\
	\vdots & & \ddots&  & \vdots & \vdots\\
	0 &  &  & I_{\ell_{\mu}} & 0 &  \\
	0 &0 & \ldots &  0 & 0 &  \\
	0 & & \ldots &   &  &  0\\
	 \end{bmatrix}, \quad (N^{(Ec)})^TN^{(Ec)} =\begin{bmatrix}
		0 & 0 \\
		 0 & \quad I_{\sum_{i=2}^{\mu} \ell_i}  
	
	 \end{bmatrix}.
	\]
Therefore,
\begin{eqnarray*}
K(N^{(Ec)})&=& (N(N^{(Ec)})^T + (I-N^{(Ec)}(N^{(Ec)})^T))(N^{(Ec)})\\
&=& N(N^{(Ec)})^T(N^{(Ec)}) =N.
\end{eqnarray*}
\end{proof}

To construct of the matrix functions $N^{(k)}$ we proceed as follows.

\begin{itemize}
\item[(1)] We define a nonsingular block upper  triangular matrix function $K^{(0)}$ accordingly to Lemma \ref{lem:construct_Kc}
\[
K^{(0)}:=N^{(0)}(N^{(Ec)})^T + (I-N^{(Ec)}(N^{(Ec)})^T).
\]

Note that constructing $K^{(0)}$ we rearrange  the columns of $N^{(0)}$ in such a way that all nonsingular $R_{i+1,i+1} \in \Real^{\ell_{i+1} \times \ell_{i+1}}$  from the secondary diagonal blocks $N^{(0)}_{i,i+1}$, $k=2, \ldots, \mu$, are in the diagonal and the remaining diagonal elements are one, cf.\ Figure \ref{fig:IntermediateSteps}. Therefore, for the diagonal blocks of $K^{(0)}$ it holds
\[
K^{(0)}_i = \begin{bmatrix}
	R_{i+1,i+1} & 0 \\
	0 & I_{\ell_{i}-\ell_{i+1}}
\end{bmatrix} :\mathcal I\rightarrow \Real^{ \ell_i \times \ell_i} , \quad i=1, \ldots, \mu-1
\]
and
\[
K^{(0)}_{\mu}=I_{m-\sum_{i=1}^{\mu-1}\ell_{i}},
\]
leading to
\[
(K^{(0)}_i)' = \begin{bmatrix}
	(R_{i+1,i+1})' & 0 \\
	0 & 0
\end{bmatrix}, \quad i=1, \ldots, \mu-1, \quad (K^{(0)}_{\mu})' =0 \in \R^{(m-\sum_{i=1}^{\mu-1}\ell_{i}) \times (m-\sum_{i=1}^{\mu-1}\ell_{i})}.
\]
By construction, $K^{(0)}$ is nonsingular, the inverse  $(K^{(0)})^{-1}$ has the same pattern than $K^{(0)}$ and due to Lemma \ref{lem:construct_Kc} it fulfills
\begin{eqnarray*}
(K^{(0)})^{-1}N^{(0)}=N^{(E_c)}.
\end{eqnarray*}

\begin{figure}[htbp]
    \begin{minipage}[b]{\textwidth}
        \includegraphics[width=\textwidth]{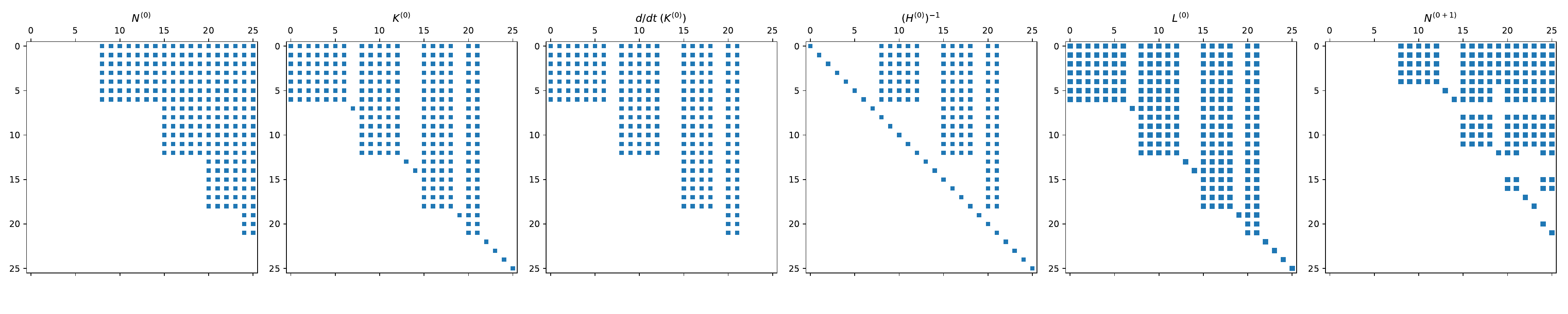}
    \end{minipage}\\
    \hfill
    \begin{minipage}[b]{\textwidth}
        \includegraphics[width=\textwidth]{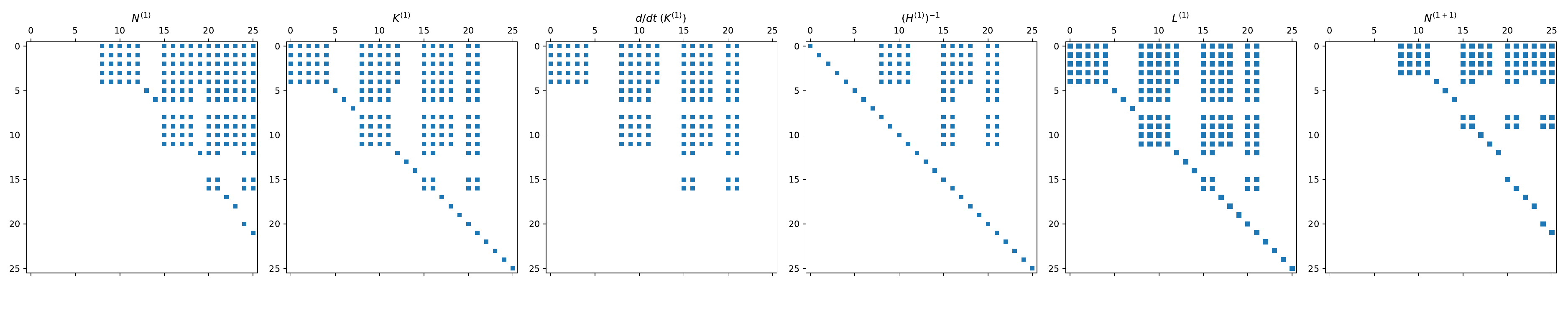}
    \end{minipage}\\
		\begin{minipage}[b]{\textwidth}
        \includegraphics[width=\textwidth]{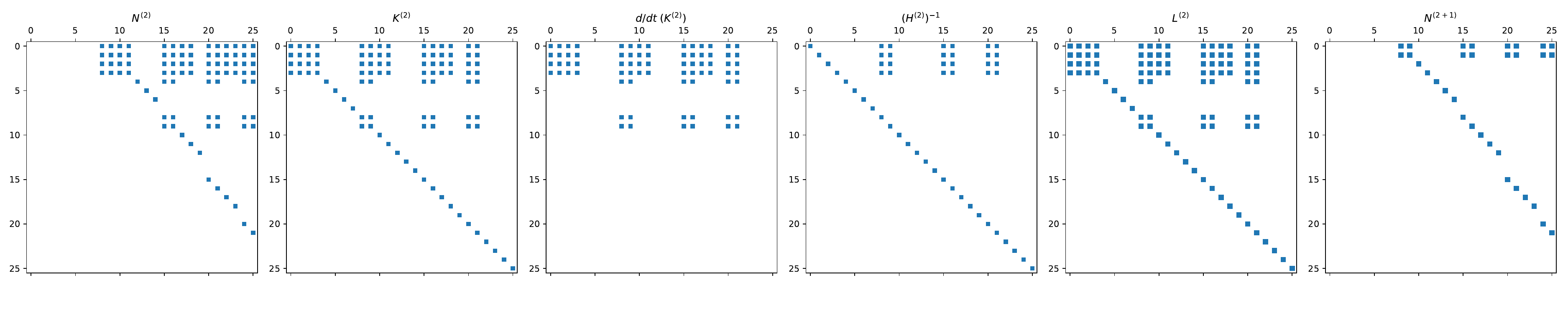}
    \end{minipage}\\
    \hfill
    \begin{minipage}[b]{\textwidth}
        \includegraphics[width=\textwidth]{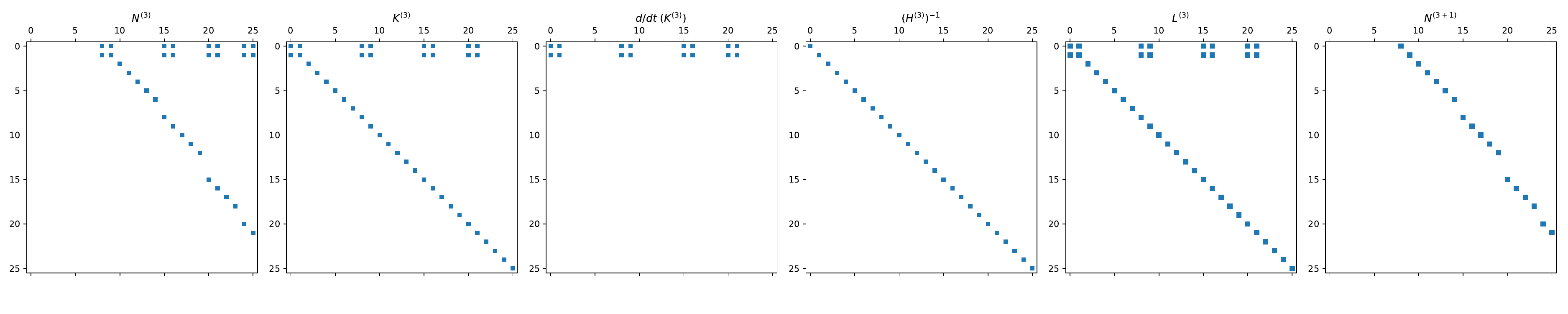}
    \end{minipage}\\
		\caption{Structure of matrices in intermediate steps in Section \ref{sec:Step1_col}}
    \label{fig:IntermediateSteps}
\end{figure}

Moreover, the particular structure of the blocks of $K^{(0)}$ imply important properties that are crucial  using  Lemma \ref{lemma_k_h_l} in this context, see Figure \ref{fig:IntermediateSteps}.
\begin{itemize}
\item For $N^{(0)}$: 
For the integer
\[
\kappa^{(0)}:=m-\sum_{i=1}^{\mu-1} \ell_{i}
\]
at least the last $\kappa^{(0)}$ rows of $N^{(0)}$ and $N^{(Ec)}$ coincide. Indeed, for $\ell_{\mu-1}> \ell_{\mu}$ also some more, but this is not crucial in later steps, so that for the sake of simplicity we will not emphasize it.
	\item For $K^{(0)}$: 
	At least the last $\kappa^{(0)}$  rows $K^{(0)}$  and $I$ coincide. Also some of the last columns, but the columns are not decisive in this context.
\item For $(K^{(0)})'$:\\ 
At least the last $\kappa^{(0)}$ rows of  $(K^{(0)})'$ are zero (as well as some others, that are not of importance for the steps below).
\item For $(K^{(0)})^{-1}N^{(0)} K^{(0)}  =N^{(E_c)}K^{(0)} \in SUT_{columns}$: \\
Note that on the one hand, at least the last $\kappa^{(0)}$ rows of $K^{(0)}$ and $I$ coincide.  On the other hand, the multiplication with $N^{(E_c)}$ from the left is a shift of blocks, such that at least
the last $\kappa^{(0)}+\ell_{\mu-1}$ rows  of $N^{(E_c)}K^{(0)}$ and  $N^{(E_c)}$ coincide.
 
\item For $(K^{(0)})^{-1}N^{(0)}(K^{(0)})' = N^{(E_c)}(K^{(0)})' \in SUT$:\\
This expression is crucial for non-constant coefficients. In this case, at least the last $\kappa^{(0)}$ rows of $(K^{(0)})'$ are zero.  Again, the multiplication with $N^{(E_c)}$ from the left is a shift of blocks, such that at least
the last $\kappa^{(0)}+\ell_{\mu-1}$ rows of $N^{(E_c)}(K^{(0)})'$ are zero.
\item For $H^{(0)}=I+(K^{(0)})^{-1}N^{(0)}(K^{(0)})'$:\\
 At least the last  $\kappa^{(0)}+\ell_{\mu-1}$ rows of $H^{(0)}$ and $I$ coincide. $H^{(0)}$ is nonsingular.
\item For $(H^{(0)})^{-1}N^{(E_c)}$:\\
 Since the last $\kappa^{(0)}+\ell_{\mu-1}$ rows of the matrix  $\bar{N}^{(0)} :=-(K^{(0)})^{-1}N^{(0)}(K^{(0)})' \in SUT$ are zero,  the last $\kappa^{(0)}+\ell_{\mu-1}$ rows of 
\[
(H^{(0)})^{-1}N^{(E_c)} =(I-\bar{N}^{(0)})^{-1} N^{(Ec)} = N^{(Ec)}+ \bar{N}^{(0)} N^{(Ec)} + \cdots + (\bar{N}^{(0)})^{\mu-1} N^{(Ec)},
\]
coincide with that of $N^{(Ec)}$.
\end{itemize}

Let us consider now $N^{(1)}=\hat{E}=L^{(0)}EK^{(0)}$ for $E=N^{(0)}$, $L^{(0)}=(H^{(0)})^{-1}((L^{(0)})^{-1}$ from Lemma  \ref{lemma_k_h_l}:
\begin{eqnarray*}
N^{(1)} &=& (H^{(0)})^{-1}(K^{(0)})^{-1}N^{(0)} K^{(0)} = (H^{(0)})^{-1}N^{(E_c)} K^{(0)}\\
 &=& N^{(Ec)}K^{(0)}+ (\bar{N}^{(0)} N^{(Ec)} + \cdots + (\bar{N}^{(0)})^{\mu-1} N^{(Ec)})K^{(0)},
\end{eqnarray*}
with the blocks
\begin{eqnarray*}
(N^{(1)})_{i,i+1}&=&(N^{(Ec)}K^{(0)})_{i,i+1} =\begin{bmatrix}
	I_{\ell_{i+1}}\\
	0
\end{bmatrix}\begin{bmatrix}
	R_{i+2,i+2} & 0 \\
	0 & I_{\ell_{i+1}-\ell_{i+2}}
\end{bmatrix}\\
 &=& \begin{bmatrix}
	R_{i+2,i+2} & 0  \\
	0 & I_{\ell_{i+1}-\ell_{i+2}}  \\
	0 & 0 
\end{bmatrix} :\mathcal I\rightarrow \Real^{ \ell_i \times \ell_{i+1}} , \quad i=1, \ldots, \mu-2,
\end{eqnarray*}
that have full column rank, and
\[
(N^{(1)})_{\mu-1,\mu} =  \begin{bmatrix}
	I_{\ell_{\mu}}   \\
	0 
\end{bmatrix}=(N^{(Ec)})_{\mu-1,\mu}.
\]

Summarizing, we constructed $N^{(1)}\in SUT_{columns}$ such that
\[
\{N^{(0)}, I \} \quad \sim \quad  \{N^{(1)}, I \}
\]
and at least the last  
\begin{eqnarray*}
\kappa^{(1)}&:=&\kappa^{(0)} +\ell_{\mu-1} =m-\sum_{i=1}^{\mu-2} \ell_{i}
\end{eqnarray*}
rows of $N^{(1)}$ and $N^{(Ec)}$ coincide.

\item[(2)] For $k=1, \ldots, \mu-1$, we repeat this procedure starting with $N^{(k)}$. Since at least the last $\kappa^{(k)}$ rows of  $N^{(k)}$ and $N^{(Ec)}$ coincide, for 
	\[
	K^{(k)} =N^{(k)}(N^{(Ec)})^T + (I-N^{(Ec)}(N^{(Ec)})^T) 
	\]
	at least the last 
	\[
	\kappa^{(k)}=m-\sum_{i=1}^{\mu-(k+1)} \ell_{i}
	\]
 columns and rows of $K^{(k)}$ and $I$ coincide and we can proceed analogously as in the first step, obtaining
a nilpotent matrix $N^{(k+1)}=L^{(k)}N^{(k)}K^{(k)} \in SUT_{columns}$ with the property that at least the last
\begin{eqnarray*}
\kappa^{(k+1)}&:=&\kappa^{(k)} + \ell_{\mu-k}
=  m-\sum_{i=1}^{\mu-(k+2)} \ell_{i}
\end{eqnarray*}
rows of $N^{(k+1)}$ and $N^{(E_c)}$ coincide. Therefore, $N^{(\mu-1)}= N^{(E_c)}$.
\end{itemize}

Note that we generated step by step zeros above all secondary diagonal blocks due to two effects:
\begin{itemize}
	\item On the one hand, step by step the  $N^{(Ec)}K^{(k-1)}$ will contain zeros columns and rows from bottom right to top left, cf.\ Figure \ref{fig:Step_dot_zero}.
	\item On the other hand, $\bar{N}^{(k-1)}$ explains that some right columns
  of $N^{(k)}$ and $N^{(Ec)}$  do not coincide for $N' \neq 0$. However,  these entries will also become zero from bottom to top with the iterations, cf.\ the difference between Figures \ref{fig:Step_dot_zero} and \ref{fig:Steps_dot_not_zero}.
\end{itemize}

\section{Proof with full-row rank assumption} \label{sec:rows}

If for integers $ \ell_{1} \leq \ldots \leq \ell_{\mu}$ we define elementary matrices from $SUT_{rows}$ by
\[
N^{(E_r)} =
	\begin{bmatrix}
   0&N^{(E_r)}_{12}&0&\cdots&0 \\
   &0&N^{(E_r)}_{23}& 0&0\\
   &&\ddots&\ddots&\vdots\\
   &&&&N^{(E_r)}_{\mu-1, \mu}\\
   &&&&0
   \end{bmatrix} \in \Real^{m\times m}
\]
with blocks
\begin{equation}
N^{(E_r)}_{i,i+1} = \begin{bmatrix}
	0 & I_{\ell_{i}} 
\end{bmatrix} \in \Real^{\ell_{i} \times \ell_{i+1}}. \label{BlocksIrow}
\end{equation}
This structure is visualized in Figure \ref{fig:Powers_Erows}.
\begin{figure}
\includegraphics[width=0.8\textwidth]{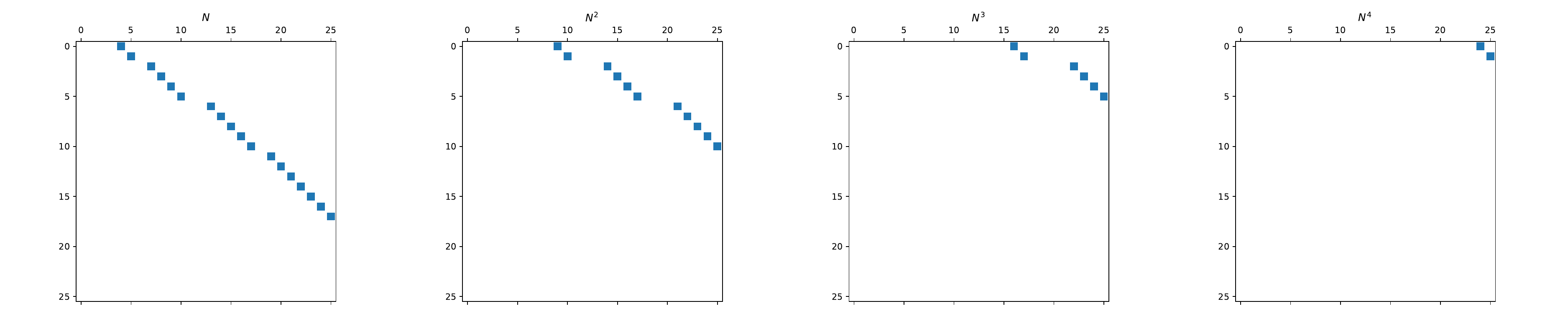}
\caption{Powers of $N=N^{(Er)}$ with $N^5=0$ for $N$ with full row-rank secondary blocks, $m=26$,  $r=18$, $\theta_0= 7,  \theta_1= 5,  \theta_2= 4,  \theta_3=2$.}
\label{fig:Powers_Erows}
\end{figure} 

\begin{proposition}\label{pro:N_row}
For every matrix function $N \in SUT_{rows}$ it holds

\[
\{N,I\} \quad \sim \quad \{N^{(E_r)},I \}.
\]
\end{proposition}

Of course, this proposition follows already from Proposition \ref{pro:N_col}. But for completeness, here we also outline the steps of a direct proof.

\begin{proof}
Analogously as in the last section, for the proof we first  make a rearrangement in an initial step
\[
\{N,I\} \quad \sim \quad \{N^{(0)},I\}
\]
and perform then  a finite iteration
\[
  \{N^{(0)},I \}  \quad \sim \quad \{N^{(1)},I \}  \quad \sim \quad \ldots \quad \sim \quad \{N^{(\mu -1)},I \}  \quad =    \quad \{N^{(E_c)},I\}.
\] 
These steps are analogous to the ones in the last section, but in a certain sense the other way around. We will only present the main differences.
\end{proof}

\subsection{ \textbf{Step 0:} Obtaining quadratic nonsingular blocks by smooth  SVD}
The aim of this step is to construct a matrix function $N^{(0)} \in SUT_{columns}$ such that 
	\[
 \{N,I\} \quad \sim \quad \{N^{(0)},I \} \quad \mbox{and} \quad 	N^{(0)}_{i,i+1} =\begin{bmatrix}
		0 & R_{i,i}
	\end{bmatrix}:\mathcal I\rightarrow \Real^{\ell_{i}\times \ell_{i+1} }
	\]
	for nonsingular matrix functions $R_{i,i}$, $i=1, \ldots, \mu-1$ . \\

Using the singular value decomposition we define
\begin{itemize}
	\item $N_{i,i+1}  = U_{i} \Sigma_i V_{i}^T$, 
	\item $B_V$: Block diagonal matrix function containing $I_{\ell_{1}}, V_2, \ldots, V_{\mu}$ with $B_V^{-1}=B_V^T$.
\end{itemize}
and obtain secondary diagonal blocks
\[
(\hat{N})_{i,i+1}=V_{i-1}^TU_{i} \Sigma_i V_{i}^TV_i=
\begin{bmatrix}
	 \tilde{R}_{i,i} & 0
	\end{bmatrix}.
\]
Therefore, we perform a further equivalence transformation with a block matrix containing the permutation matrices $P_{i}$ that invert the order of the columns and rows inside the blocks to obtain $N^{(0)}$:
\[
P_{i}\begin{bmatrix}
	 \tilde{R}_{i,i} & 0
	\end{bmatrix} P_{i+1}=: \begin{bmatrix}
	 0 & R_{i,i} 
	\end{bmatrix}.
\]

\begin{figure}[htbp]
    \begin{minipage}[b]{\textwidth}
        \includegraphics[width=\textwidth]{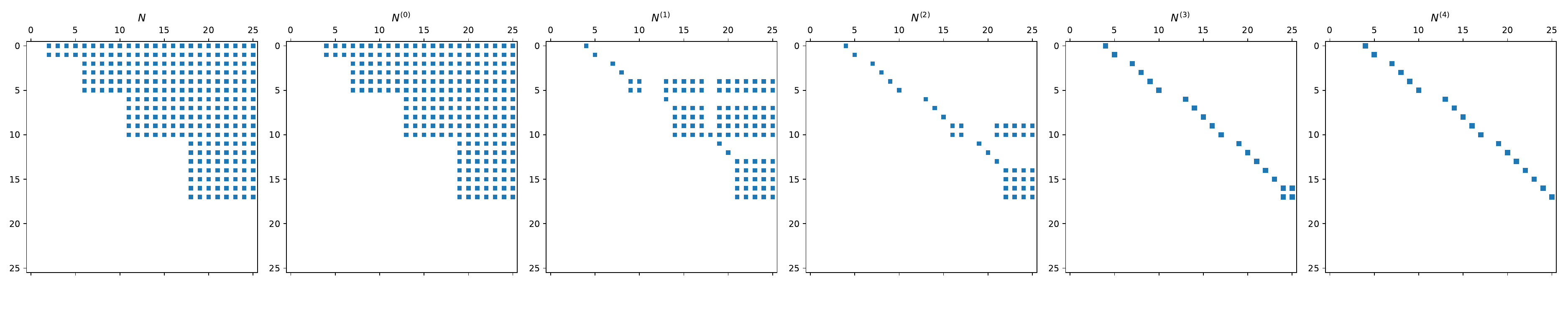}
        \caption{Visualization of  $N^{(k)}=L^{(k-1)}N^{(k-1)}K^{(k-1)}$ for $ \mu$ steps starting from $N$ (left) and ending with $N^{(Er)}$ for  $N'= 0$.}
        \label{fig:Step_dot_zero_rows}
    \end{minipage}\\
		\begin{minipage}[b]{\textwidth}
        \includegraphics[width=\textwidth]{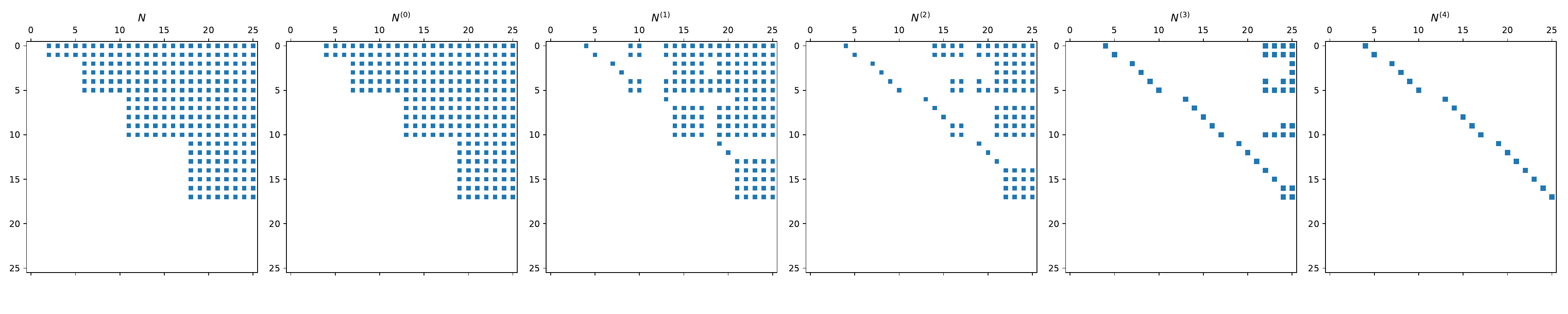}
        \caption{Visualization of  $N^{(k)}=L^{(k-1)}N^{(k-1)}K^{(k-1)}$ for $ \mu$ steps starting from $N$ (left) and ending with $N^{(Er)}$ for $N'\neq 0$.}
        \label{fig:Steps_dot_not_zero_rows}
    \end{minipage}
\end{figure}

\subsection{ \textbf{Step 1:} Obtaining $N^{(E_r)}$  }
The aim of this step is  to construct $N^{(E_r)} \in SUT_{columns}$ by a finite iteration such that  
\[
 \quad \{N^{(0)},I \}  \quad \sim \quad \{N^{(i)},I \}  \quad \sim \quad \ldots \quad \sim \quad \{N^{(\mu-1 )},I \}  \quad =    \quad \{N^{(E_r)},I\}.
\]

In this case, we will use the following lemma several times.
\begin{lemma}\label{lem:construct_Kr}
	For any $N \in SUT_{rows}$, a corresponding $N^{(Er)}\in SUT_{rows}$ and 
	
	\[
	K:=(N^{(Er)})^TN + (I-(N^{(Ec)})^T N^{(Ec)}) 
	\]
	it holds
	\[
	(N^{(Er)})K=N.
	\]
\end{lemma}
\begin{proof}
By construction we have
	\[
	 (N^{(Er)})^T N^{(Er)}=\begin{bmatrix}
	0 &   & \cdots  &  &   &0\\
	& 0 & 0  & \cdots  & 0 &   \\
	& 0 & I_{\ell_1} & & 0&    \\
	\vdots & & \ddots&  &   \vdots \\
	& 0 &  &  & 0 & 0   \\
	0& 0 &0 & \ldots &  0 & I_{\ell_{\mu-1}} \\
	
	 \end{bmatrix}, \quad N^{(Er)}(N^{(Er)})^T =\begin{bmatrix}
		 I_{\sum_{i=1}^{\mu-1} \ell_i} & 0 \\
		 0 & 0
	 \end{bmatrix}.
	\]
Therefore,
\begin{eqnarray*}
N^{(Ec)}K&=&N^{(Ec)}((N^{(Er)})^TN + (I-(N^{(Ec)})^T N^{(Ec)})) \\
&=&N^{(Ec)}((N^{(Er)})^TN =N.
\end{eqnarray*}
\end{proof}

This time, we start from the fact that the first columns of $N^{(0)}$ and $N^{(E_r)}$ coincide and construct a series of $N^{(k)}$ 
that becomes $N^{(E_r)}$ from left to right, see Figures \ref{fig:Step_dot_zero_rows} and \ref{fig:Steps_dot_not_zero_rows}, using again Lemma \ref{lemma_k_h_l}.
Indeed, for constructing $K^{(k)}$ we rearrange  the rows of $N^{(k)}$ in such a way that all nonsingular $R_{i,i}$ are in the diagonal and the remaining diagonal elements are one.

\section{Conclusion}\label{sec:sum}

The conjecture is proven, since for a matrix $J$ in Jordan form with the blocks described in Conjecture \ref{conjecture} there exist permutation matrices $P_c$ and $P_r$ such that
\[
J=P_cN^{(Ec)}P_c^T = P_rN^{(Er)}P_r^T.
\]
Therefore, shortly written, $\{E,F\}$ is regular iff
\[
\{E,F\} \ \sim \ \left\{ \begin{bmatrix}
	I & 0 \\
	0 & N^{(Ec)}
\end{bmatrix}, \begin{bmatrix}
	\Omega & 0 \\
	0 & I
\end{bmatrix} \right\} \ \sim \ \left\{ \begin{bmatrix}
	I & 0 \\
	0 & N^{(Er)}
\end{bmatrix}, \begin{bmatrix}
	\Omega & 0 \\
	0 & I
\end{bmatrix} \right\} \ \sim \ \left\{ \begin{bmatrix}
	I & 0 \\
	0 & J
\end{bmatrix}, \begin{bmatrix}
	\Omega & 0 \\
	0 & I
\end{bmatrix} \right\}.
\]

Even more generally, $\{E,F\}$ is regular iff 
\[
\{E,F\} \ \sim \ \left\{ \begin{bmatrix}
	I_d & 0 \\
	0 & N
\end{bmatrix}, \begin{bmatrix}
	\Omega & 0 \\
	0 & I_{n-d}
\end{bmatrix} \right\} 
\]
for a constant $r=\rank E$ and a constant nilpotent matrix $N \in \Real^{(n-d)\times (n-d)}$ of index $\mu$ with
\[
\theta_i=\rank N^{i+1} - \rank N^{i+2}, \quad i=0, \ldots, \mu-2, \quad d=r-\sum_{i=0}^{\mu-2} \theta_i.
\]


\bibliography{Regularity-SSCF_DAEs}{}
\bibliographystyle{plain}

\end{document}